\documentclass[a4paper,11pt]{article}
\usepackage[utf8]{inputenc}
\usepackage[T1]{fontenc}
\usepackage[english]{babel}
\usepackage{amssymb}
\usepackage[dvipsnames]{xcolor}
\usepackage{amsmath}
\usepackage{amsfonts}
\usepackage{comment} 
\usepackage{bbm}
\usepackage{color}
\usepackage{amsthm}
\usepackage{fancyhdr}
\usepackage{esint}
\usepackage{authblk}
\usepackage{multicol}
\usepackage[cm]{fullpage}
\usepackage{parskip}

\usepackage[mathscr]{euscript}   
\usepackage[color,all]{xy}      
\usepackage{graphicx}
\usepackage[normalem]{ulem}	
\usepackage{tikz}

\usepackage{enumerate}

\newtheorem{exemple}{Example}
\numberwithin{equation}{section}

\author[1]{Nathan Couchet}
\affil[1]{Université Clermont Auvergne, CNRS, LMBP, F-63000 Clermont-Ferrand, France\\ \texttt{nathan.couchet@uca.fr}}

\author[2]{Robert Yuncken}
\affil[2]{Institut Élie Cartan de Lorraine, Site de Metz 3, rue Augustin Fresnel Technopole Metz 57000 Metz\\ \texttt{robert.yuncken@univ-lorraine.fr}}

\title{On polyhomogeneous symbols and the Heisenberg pseudodifferential calculus.}
\date{}

\usepackage{xcolor}
\usepackage[colorlinks=true,linkcolor=blue,urlcolor=blue,citecolor=blue]{hyperref}

\newcommand{\R}{\mathbb{R}}

\newcommand{\V}{\mathcal{V}}

\newcommand{\m}{\medbreak} 
\newcommand{\bg}{\bigbreak}
\newcommand{\N}{\mathbb{N}}
\newcommand{\x}{\xi}

\newcommand{\supp}{\mathrm{supp}}

\newcommand{\Hn}{\mathbb{H}_{n}}

\newcommand{\Hom}{\mathrm{Hom}}

\newcommand{\Exp}{\mathbb{E}\mathrm{xp}^{\overline{X}} }

\newcommand{\TM}{\mathbb{T} M}

\newcommand{\THM}{\mathbb{T}_HM}
\newcommand{\kgt}{\tilde{\mathbbm{k}}}
\newcommand{\kg}{\mathbbm{k}}
\newcommand{\Gu}{G^{(0)}}

\newtheorem{theoreme}{Theorem}[section]
\newtheorem{proposition}[theoreme]{Proposition}
\newtheorem{lemma}[theoreme]{Lemma}

\theoremstyle{definition}

\newtheorem{definition}[theoreme]{Definition} 

\begin{document}

\maketitle

 

 
\begin{abstract}

Polyhomogeneous symbols, defined by Kohn-Nirenberg \cite{Kohn1965PsiDOalgebra} and Hörmander \cite{hormander1985analysisIII} in the 60's, play a central role in the symbolic calculus of most pseudodifferential calculi.  We prove a simple characterisation of polyhomogeneous functions which avoids the use of asymptotic expansions.  Specifically, if $U$ is open subset of $\mathbb{R}^d$, then a polyhomogeneous symbol on $U\times\mathbb{R}^d$ is precisely the restriction to $t=1$ of a function on $U\times\mathbb{R}^{d+1}$ which is homogeneous for the dilations of $\mathbb{R}^{d+1}$ modulo Schwartz class functions.  This result holds for arbitrary graded dilations on the vector space $\mathbb{R}^d$.  As an application, using the generalisation $\mathbb{T}_HM$ of A. Connes' tangent groupoid for a filtered manifold $M$, we show that the Heisenberg calculus of Beals and Greiner \cite{Beals2016Heisenbergcalculus} on a contact manifold or a codimension $1$ foliation coincides with the groupoid calculus of \cite{Yuncken2019groupoidapproach}.

\end{abstract}

\section{Introduction: From a theorem on polyhomogeneous symbols to Beal and Greiner's pseudodifferential calculus on Heisenberg manifold.}
\label{sec:introduction}

The notion of a pseudodifferential operator is rather general.  The key idea, due to Kohn and Nirenberg, is to multiply the Fourier transform of a function by a symbol function, then apply the inverse Fourier transform.
In order for this procedure to be well-defined, and further to ensure that the class of pseudodifferential operators has nice properties such as stability under compositions and coordinate changes, one must impose strong restrictions on the class of symbols permitted.  

For this reason, one typically restricts to the class of polyhomogeneous symbols, or some generalization.
These are the symbols which admit an asymptotic expansion:
\[
  a(x,\xi) \sim \sum_{j=0}^\infty a_j(x,\xi),
\]
with each function $a_j$ being smooth and homogeneous in $\xi$ of order $m-j$ on $\R^n \times (\R^d \setminus\{0\})$,
see \cite[Definition 18.1.5]{hormander1985analysisIII}.   The notation $\sim$ here signifies that $a$ differs from the $N$th partial sum of the series by a symbol-class function of order $m-N-1$ outside of a neighbourhood of $\R^n\times\{0\}$.

The notion of polyhomogeneous symbol is powerful, but somewhat cumbersome.  
For instance, the proof that polyhomogeneous pseudodifferential operators are stable under smooth changes of coordinates is relatively technical, and it becomes a serious difficulty when one considers generalizations to more complicated calculi \cite{Beals2016Heisenbergcalculus, Taylor1984microlocalanalysis, melin1982lie}.

The simplest example of a polyhomogeneous symbol is a function which is polynomial in $\xi$, where we have an exact equality:
\begin{equation}
\label{eq:polynomial}
  a(x,\xi) = \sum_{j=0}^m a_j(x,\xi).
\end{equation}
We have a simple way of converting an arbitrary polynomial into a homogeneous polynomial by weighting each $a_j$ by an extra variable.  This can be stated as follows.

\begin{proposition}  \label{96}
Let $a\in C^\infty(\R^n \times\R^d)$.  The following are equivalent:
\begin{enumerate}
\item $a$ is an (inhomogeneous) polynomial in the variable $\xi$ of order $\leq m$. 
\item There exists $u \in  C^\infty(\R^n \times \R^d \times \R)$ a homogeneous polynomial of order $m$ in the variables $(\xi,t)$, such that $u(x,\xi,1)=a(x,\xi)$.
\end{enumerate}
\end{proposition}

\begin{proof}
  With the notation of Equation \eqref{eq:polynomial}, it suffices to put $u(x,\xi) =  \sum_{j=0}^m t^j a_j(x,\xi)$.
\end{proof}

It turns out that there is an analogous theorem for polyhomogeneous functions.  We say that a function $u\in C^\infty(\R^n \times\R^d\times\R)$ is \emph{homogeneous modulo Schwartz} of order $m$ in $(\xi,t)$ if, for all $s\in\R^*_+$ the function
\(
  (x,\xi,t) \mapsto u(x,s\xi,st) - s^mu(x,\xi,t)  
\)
is Schwartz-class in the variables $(\xi,t)$ (see Definition \ref{8} for the precise definition).

\begin{theoreme}
\label{intro_polyhomogeneous}
Let $a\in C^\infty(\R^n \times\R^d)$.  The following are equivalent:
\begin{enumerate}
\item $a$ is a polyhomogeneous symbol of order $\leq m$.
\item There exists $u\in C^\infty(\R^n \times\R^d\times\R)$, homogeneous modulo Schwartz of order $m$ in the variables $(\xi,t)$, such that $u(x,\xi,1)=a(x,\xi)$.
\end{enumerate}
\end{theoreme}

This theorem, and its generalisation in the context of graded dilations on $\R^d$, is the first main result of the present paper.  For the statement in the presence of graded dilations, see Theorem \ref{2}.

One reason to be interested in Theorem \ref{intro_polyhomogeneous} is that it can be used to prove that the definition of classical polyhomogeneous pseudodifferential operators, due to Kohn-Nirenberg \cite{Kohn1965PsiDOalgebra} and developed notably by H\"ormander \cite{hormander1985analysisIII}, coincides with recent definitions using Connes' tangent groupoid based on work by Debord-Skandalis \cite{debord2014adiabatic}. The groupoid approach has certain advantages over the classical approach, in particular because the stability under compositions and changes of coordinates are obtained naturally from the structure of the tangent groupoid.

The equivalence of the two approaches for the classical unfiltered calculus was proven both by Debord-Skandalis \cite{debord2014adiabatic} and Van Erp with the second author \cite{Yuncken2019groupoidapproach} in different ways.  In this article, we are aiming to systematise the relationship between symbolic calculi and the groupoid calculus.

As an application of Theorem \ref{intro_polyhomogeneous}, we will prove the analogous result for the Heisenberg calculus on a contact manifold.

\begin{theoreme}
\label{contact_manifold}
Let $M$ be a contact manifold.  Beals and Greiner's algebra of Heisenberg pseudodifferential operators on $M$ \cite{Beals2016Heisenbergcalculus} coincides with the algebra of pseudodifferential operators defined via the Heisenberg tangent groupoid in \cite{Yuncken2019groupoidapproach}.
\end{theoreme}

The Heisenberg tangent groupoid was defined independently by Van Erp \cite{VanErp2010AtiyahSingerI} and Ponge \cite{ponge2006tangentgroupoid}, see also \cite{VanErp2017groupoid, choi2019tangent, Higson2018Euler, Mohsen2021groupoid}.  

We also obtain the analogous theorem for codimension $1$ foliations, see Theorem \ref{foliations}.  We expect that the groupoid calculus of \cite{Yuncken2019groupoidapproach} coincides with the Heisenberg calculus of \cite{Beals2016Heisenbergcalculus} even when the osculating groups are non-constant, and indeed that it coincides with Melin's calculus \cite{melin1982lie} for arbitrary filtered manifolds.  However, the proof in these cases becomes far more technical due to the delicate nature of privileged coordinate systems, see \cite{ponge2008heisenberg, choi2019tangent}.  This is outside the scope of the present paper.

\bg
\bg
\bg
\bg
\bg
\bg
\bg

\subsection{Homogeneous, polyhomogeneous, homogeneous functions modulo Schwartz with respect to dilations and the main theorem.}
In this section we introduce the notation and terminology we will need for the rest of the article.
We consider the one-parameter family of dilations, on $\R^d$:

\begin{equation} \label{3}
\delta_s : \R^d \rightarrow \R^d, (\xi_1,...,\xi_d) \mapsto (s^{\rho_1} \xi_1,...,s^{\rho_d} \xi_d), ~~ s \in \R_+^*,
\end{equation}

where the d-tuple $(\rho_1,...,\rho_d) \in (\N^*)^d$. We consider also the dilations on $\R^{d+1}$:

\begin{equation} \label{4}
\tilde{\delta}_s : \R^d \times \R \rightarrow \R^{d+1}, (\xi,t) \mapsto (\delta_s(\xi),st),
\end{equation}

where we augment $\delta_s$ by one dimension with weight  $\rho_{d+1}=1$. 

The notions of homogeneous and polyhomogeneous functions (consult classical references : \cite{Taylor1982PsiDo}, \cite{Treves1980PsiDointroduction}, \cite{Petersen1983PsiDO} and more recent books like  \cite{abels2011pseudodifferential} chapter I.3, \cite{Wong2014PsiDO} ) admit generalisations in the presence of graded dilations as in \eqref{3}. We begin with the definition and properties of a homogeneous quasi-norm $|.|$,  following for instance the book \cite{FischerRuzhansky2016Quantization} by Fischer and Ruzhansky, see section 3.1.6.
\begin{definition}
A homogeneous quasi-norm is a function $|.| : \R^d \rightarrow \R_+$ satisfying:
\begin{enumerate}
\item{Symmetry: $\forall ~ x \in \R^d $, $|-x|=|x|$.}
\item{Homogeneity: $\forall ~ x \in \R^d, \forall ~ s>0$ $|\delta_s x|=s|x|$.}
\item{Positivity: $|x|=0 \Leftrightarrow x=0$.}
\item{Triangle inequality: there exists $C>0$ such that $\forall ~ (x,y) \in \R^d$ we have:
$$|x+y| \leq C(|x|+|y|).$$}
\end{enumerate}

\end{definition}
In what follows, $|.|$ refers to a  homogeneous quasi-norm on $\R^d$. For instance we could set $|\xi|:=\sum_{k=1}^{d} |\xi_k|^{\frac{1}{\rho_k}}$ or $|\xi|:=(\sum_{k=1}^{d} |\xi_k|^{\frac{2a}{\rho_k}})^{\frac{1}{2a}}$, where $a=lcm(\rho_j)$.
The results in \cite{FischerRuzhansky2016Quantization} show that  homogeneous quasi-norms behave mostly like norms on $\R^d$. Any two homogeneous quasi-norms on $\R^d$ are equivalent in the usual sense. Finally if $||.||$ is the euclidian norm on $\R^d$, then there exists $a,b,C_1,C_2 > 0$ such that $\forall ~ x \in \R^d$:

\begin{equation} \label{148}
C_1||x||^a \leq |x| \leq C_2 ||x||^b.
\end{equation}

We can now move on to  the definitions of the spaces of functions involved in this article. In the following, $U$ is an open subset of $\R^n$. 

\begin{definition} \label{8}
We write $\mathcal{H}_{G}^m(U \times \R^d)$ for the set of  homogeneous functions of order $m$. This consists of smooth functions on $U \times (\R^{d} \setminus \{ 0 \})$ that satisfy:

\begin{equation}
\forall ~  \xi \in \R^d, \forall ~ x \in U, \forall ~ s \in \R_+^*, ~~~~ f(x, \delta_s \xi)=s^m f(x,\xi),
\end{equation}

Sometimes we will say that $f$ is homogeneous of order $m$ on the nose. Note that despite the notation, elements of $\mathcal{H}_{G}^m(U \times \R^d)$ are only defined on $U \times (\R^{d} \setminus \{ 0 \})$.

\end{definition}

\begin{definition} \label{9} 
We write $\mathcal{S}_G(U \times \R^d)$ for the set of smooth functions which are Schwartz class in the variable  $\R^d$. More precisely, it consists of functions $f \in \mathcal{C}^{\infty}(U \times \R^{d})$ which satsify that for every compact set $K$ in $U$, for every $k \in \mathbb{N}$ and for every pair of multi-indices $(\alpha, \beta)$, there exists a constant $C_{\alpha,\beta,k,K} >0$ such that $\forall ~ \xi \in \R^d $, $\forall ~ x \in K$:

\begin{equation} \label{149}
|D_x^{\alpha} D_\xi^{\beta} f(x,\xi)| \leq C_{\alpha,\beta,k,K} (1 + |\xi|)^{-k}.
\end{equation}
\end{definition}
\textbf{Remark:} 
The choice of using the homogeneous quasinorm $|\xi|$ instead of the usual euclidean norm $||\xi||$ has no effect, other than a change in the constant $C_{\alpha,\beta,k,K}$,  thanks to the inequalities \eqref{148} and \eqref{149}.

\begin{definition} \label{10} 
We write $ \mathcal{H} \mathcal{S}_G^m(U \times \R^d)$  for the set of homogeneous functions of order $m$ modulo Schwartz. It consists of functions $u \in \mathcal{C}^{\infty}(U \times \R^{d})$  satisfying for all $ s \in \R_+^* $:
\begin{equation} \label{150}
(x,\xi) \mapsto u(x,\delta_s \xi) - s^m u(x,\xi) \in \mathcal{S}_G(U \times \R^d).
\end{equation}
\end{definition}

To show that $b \in \mathcal{H} \mathcal{S}_G^m(U \times \R^d)$, it suffices to show \eqref{150}, for all $s \in [1,2]$. This is true since the set of all $s$ for which \eqref{150} is true is a subgroup of $\R_+^*$. 


\textbf{Remark:} In particular, every Schwartz function from definition \ref{9} is trivially a homogeneous function of order $m$ modulo Schwartz. 
We also have the following lemma, whose proof is left to the reader:

\begin{lemma} \label{42}
Let $f \in \mathcal{H}_G^m(U \times \R^d)$ and let $\chi \in C_c^\infty(\R^d)$ satisfying $\chi =0$ in a neighbourhood  of $0$ and $\chi = 1$ outside a compact set, then $\chi f \in \mathcal{H} \mathcal{S}_G^m(U \times \R^d)$.
\end{lemma}
The following definition is the generalisation of Hörmander's and Kohn-Nirenberg's symbol class of functions to spaces with graded dilations, see  \cite{hormander1983analysis} Definition 7.8.1 p236 or  the original articles of Hörmander \cite{Hormander1967PsiDohypoellipticeq} from 1967, or of Kohn and Nirenberg \cite{Kohn1965PsiDOalgebra} from 1965.
\begin{definition} \label{11} 

We write $S_G^m(U \times \R^{d})$ for the space of symbols of order $m$. It consists of functions $a \in C^\infty(U \times \R^{d})$ which satisfy that for every compact set $K$ in $U$, for every pair of multi-indices $(\alpha,\beta)$, there exists a constant $C_{K,\alpha,\beta} >0$ such that $\forall ~ (x,\xi) \in K \times \R^{d}$ the following condition is fulfilled:

\begin{equation} \label{151}
|D_x^\alpha D_\xi^\beta a(x,\xi)| \leq C_{K,\alpha,\beta} (1+|\xi|)^{m - | \beta | },
\end{equation}

where $|\beta|$ denotes the homogeneous order of $\beta$, namely $|\beta|=\rho_1 \beta_1+...+\rho_d \beta_d$.
Moreover we set $S_G^{- \infty}(U \times \R^{d}):= \cap_{m \in \R} S_G^m(U \times \R^{d})$.
\end{definition}
\textbf{Remark:} As usual it is still true that:

\begin{equation}
S_G^{- \infty}(U \times \R^{d})=\mathcal{S}_G(U \times \R^d).
\end{equation}
The following result can be found in \cite{Taylor1982PsiDo} p37. 

\begin{proposition} \label{92}
Suppose $f \in \mathcal{H}_G^m(U \times \R^d)$ and $\phi \in C_c^\infty(\R^d)$ such that $\phi(\xi)=0$ on a compact neighbourhood of $0$ and $\phi(\xi)=1$ near infinity, then $\phi f \in S_G^m(U \times \R^d)$.
\end{proposition}

The following definition is based on the definition of classical symbol. 
The generalisation to graded dilations appears in the work of Beals and Greiner \cite{Beals2016Heisenbergcalculus}, Melin \cite{melin1982lie} and \cite{Taylor1984microlocalanalysis}, amongst others.

\begin{definition} \label{12} 
We write $S_{phg,G}^m(U \times \R^{d})$ for the set of  polyhomogeneous symbols of order $m$. It consists of functions $a \in S_G^m(U \times \R^{d})$ which admit an asymptotic expansion $a \sim \sum_{j \in \mathbb{N}} a_{j}$ with $a_{j} \in \mathcal{H}_{G}^{m-j}(U \times \R^d)$ meaning that for every integer $N$, for any compact set $K$ in $\R^d$ containing $0$ 
  :

\begin{equation}
a - \sum_{j=0}^{N} \chi_K ~ a_{j} \in S_G^{m-N-1}(U \times \R^{d}),
\end{equation}

where $\chi_K \in C_c^\infty(\R^d)$ is any smooth function such that:

\begin{equation}
\chi_K(\xi) = \left\{
    \begin{array}{ll}
        0 &  ~ \text{in} ~  K \\
        1 &  ~ \text{near  infinity.}
    \end{array}
\right.
\end{equation}
\end{definition}

\bg

We recall that the symbol of a differential operator with smooth coefficients on $U$ is always polyhomogeneous since it is polynomial in $\xi$. In light of this, Proposition   \ref{96} fuels the desire to find a more general result concerning extension of polyhomogeneous symbols. As such, we state the main result of this article.

\begin{theoreme} \label{2}
Let $a \in C^{\infty}(U \times \R^d)$ where U refers to an open subset of $\R^n$. We equip the vector spaces $\R^d$ and $\R^d \times \R$ respectively with the graded dilations $(\delta_s)$ and $(\tilde{\delta}_s)$  (see equations \eqref{3}, \eqref{4}). Then the following are equivalent:
\begin{enumerate}
\item{$a \in S_{phg,G}^m(U \times \R^d) $. }
\item{There exists $u \in  \mathcal{H} \mathcal{S}_{G}^m(U \times \R^d \times \R)$ such that $u(x,\xi,1)=a(x,\xi)$.}
\end{enumerate}
\end{theoreme}
\subsection{Outline of the paper.} 

In the second section we recall some preliminaries including an important result, Theorem \ref{1}, on the structure of homogeneous functions  modulo Schwartz. 

The next two sections are dedicated to the proof of Theorem \ref{2} :

In the third and fourth sections we prove $2) \Rightarrow 1)$ and $1) \Rightarrow 2)$ of this theorem, respectively. 

In the fifth section, we give an application to the Heisenberg pseudodifferential calculus of Beals and Greiner \cite{Beals2016Heisenbergcalculus}. In their book, they define a pseudo-differential calculus on a Heisenberg manifold, see definitions \ref{97} and \ref{112}. In \cite{Yuncken2017tangentgroupoidfilteredmanifold}, the authors used an analog of Connes' tangent groupoid,  see \cite{Yuncken2017tangentgroupoidfilteredmanifold},\cite{Yuncken2019groupoidapproach},\cite{ponge2008heisenberg}, to define pseudodifferential calculus on a filtered manifold. We will use Theorem \ref{2} to show that this groupoidal calculus coincides with Beals and Greiner's calculus in the case of the model Heisenberg manifold $M=\Hn \times \R^m$ of dimension $d+1$, with $d=2n+m$. This model manifold can be seen as an intermediary between an abelian and a non-abelian structure.

In the sixth section, we generalise this previous result to contact manifolds, using Darboux's theorem, and to foliations. 


For brevity, we omit the proofs of "common" results, but the reader will find fuller details of this article in the future memoir \cite{couchet2023thesis}.

\subsection*{Acknowledgements.}
We would like to thank Dominique Manchon, "Chargé de recherches" CNRS at UCA, Clermont Ferrand, France, for his precious advice and his kindness.

\section{Preliminaries.}
\label{sec:section1}

\subsection{ On homogeneous functions modulo Schwartz.}
\label{sec:Taylorthm}

We shall now recall a fundamental result about homogeneous functions modulo Schwartz which will be used repeatedly in what follows. We recall the notation $\mathcal{H} _{G}^m(U \times \R^d)$ and $\mathcal{H} \mathcal{S}_{G}^m(U \times \R^d)$ for functions which are homogeneous and homogeneous modulo Schwartz, respectively.   Versions of this theorem can be found in  proposition (12.72) p113 from Beals and Greiner \cite{Beals2016Heisenbergcalculus}, or in   proposition 2.2 from Taylor \cite{Taylor1984microlocalanalysis}. One will also find a revisited proof in \cite{couchet2023thesis}.

\begin{theoreme} \label{1}  
Let $u \in C^\infty(U \times \R^d)$. The following are equivalent:
\begin{enumerate}
\item{$u \in \mathcal{H} \mathcal{S}_{G}^m(U \times \R^{d})$.}
\item{For any compact neighbourhood $K$ of $0$ in $\R^d$, there exists $u' \in \mathcal{H}_G^m(U \times \R^d)$ and $u'' \in \mathcal{S}_G(U \times \R^d)$ so that, in $U \times K^c$ (where $K^c=\R^{d} \setminus K$), the following equality is true: 
$$u=u'+u''.$$}
\end{enumerate}
\end{theoreme}

We also recall the following basic fact about derivatives of homogeneous functions.

\begin{lemma} \label{15}
Let $m \in \R$ and fix a $d$-tuple $\beta:=(\beta_1,...,\beta_d)$.
\begin{enumerate}
\item{Let $u \in \mathcal{H}_G^m(U \times \R^d)$. Then $\partial_\xi^\beta u \in \mathcal{H}_G^{m-|\beta|}(U \times \R^d)$. }
\item{Let $u \in \mathcal{H} \mathcal{S}_{G}^m(U \times \R^{d})$. Then $\partial_\xi^\beta u \in \mathcal{H} \mathcal{S}_{G}^{m - |\beta|}(U \times \R^{d})$.}
\end{enumerate}
\end{lemma}
 


\section{Restricting functions which are homogeneous modulo Schwartz.}
\label{sec:section3}

\subsection{Relation with symbol class funtion.}
\label{subsection:globalstrat}

In the following sections, let $u \in \mathcal{H} \mathcal{S}_G^m(U \times \R^{d+1}) $. 
Our goal is to show that the restriction of $u$ to $U \times \R^{d} \times \{ 1 \}$ is a polyhomogeneous symbol. First, we shall show that it is a symbol class function. We begin with a lemma.



\begin{lemma} \label{16}
Let $u \in C^\infty(U \times \R^{d+1} \setminus \{0 \})$ homogeneous of order $m$ on the nose in the variable $\xi$ for the action $\tilde{\delta}_s$. Then $a(x,\xi):=u(x,\xi,1) \in S_G^m(U \times \R^d)$.
\end{lemma}

\begin{proof}
Fix  a homogeneous norm $|.|$ on $\R^d$ and the associated homogeneous norm $|(\xi,t)|=|\xi|+|t|$ on $\R^{d+1}$.  Fix also a compact set $K \subset U$. Then we have, for all $\xi \in \R^d$ and $x \in K$ :
\begin{align}
|u(x,\xi,1) |& = (|\xi|+1)^m u \Big(x,\frac{\xi}{|(\xi,1)|},\frac{1}{|(\xi,1)|} \Big) \\
& \leq C(|\xi|+1)^m,
\end{align}
where $C=max \{ |u(x,\xi,t)|, ~ x \in K, |(\xi,t)|=1 \}$. This proves the inequality \eqref{151} when $\alpha=\beta=0$. The other cases of \eqref{151} follow using Lemma \ref{15}.
\end{proof}
Thanks to this previous lemma and Theorem \ref{1} we can prove:
\begin{proposition} \label{18}
Let  $u \in \mathcal{H} \mathcal{S}_G^m(U \times \R^{d+1})$ . We set  $a_0(x,\xi)=u(x,\xi,0)$ and $a(x,\xi)=u(x,\xi,1)$.  Then we get: 
\begin{enumerate}
\item{$a_0 \in  \mathcal{H} \mathcal{S}_G^m(U \times \R^{d})$.}
\item{$a \in S_G^m(U \times \R^d)$.}
\end{enumerate}
\end{proposition}

\textbf{Remark:}
Note that the restriction of $u$ at $t=1$ is not homogeneous modulo Schwartz. This is because the dilation action $\tilde{\delta_s}$ does not restrict to $U \times \R^d \times \{1\} $.

\subsection{Construction of the asymptotic series.}
\label{subsection:induction}

We start this section with a result, which will be an useful trick in the proof of Theorem \ref{29} to come.

\begin{proposition} \label{20}
Let $f \in \mathcal{S}_G(U \times \R^d)$ and let $g \in C_c^\infty(\R)$ whose support is included in the compact interval $[a,b]$ where $a>0$. Then the function $F : (x,\xi,t) \mapsto f(x,\delta_{t^{-1}} \xi) g(t)$ belongs to $\mathcal{S}_G(U \times \R^{d+1})$.
\end{proposition}

The proof is left to the reader. We point out that the statement is credible because we multiply a cut-off function in the $t$ variable, whose support avoids $0$, with a $t$-parametrized family of $\delta_{t^{-1}}$-dilated Schwartz functions. \bg
Now we consider a smooth partition of unity $ \chi_0 + \chi_1 = 1$ on $\R$, where the smooth cut-off function $\chi_0$ defined on $\R$ satisfies:

 \begin{equation} \label{153}
\chi_0(t) = \left\{
    \begin{array}{ll}
        0 & \mbox{if } |t| \geq 2,  \\
        1 & \mbox{if } |t| \leq 1.
    \end{array}
\right. 
\end{equation}


The following result will be used several times in the future, consult proof of Proposition \ref{27}.

\begin{theoreme} \label{29}
Given $a\in \mathcal{H} \mathcal{S}^m(U \times \R^{d})$, we define:

\begin{equation}\label{23}
b(x,\xi,t):=\chi_0(t) a(x,\xi)+ \chi_1(t) |t|^m  a(x,\delta_{|t|^{-1}} \xi).
\end{equation}

Then $b \in \mathcal{H} \mathcal{S}_G^m(U \times \R^{d+1})$.

\end{theoreme}

\textbf{Remark:} For any  $ |t| \leq 1 $, note that, for any $\xi \in \R^d$ and $x \in U$, we have $b(x,\xi,t)=a(x,\xi)$.

\begin{proof}[Proof of Theorem \ref{29}]
Recall from the remark after Definition \ref{10} that it suffices to show the result for $s \in [1,2]$.
It is obvious that $t \mapsto \chi_0(st) \in C_c^\infty(\R)$ and therefore, we can claim that $(x,\xi,t) \mapsto [a(x,\delta_s \xi) - s^m a(x,\xi)] \chi_0(st) \in \mathcal{S}_G(U \times \R^{d+1})$.  We compute:
\begin{align} 
b(x,\delta_s \xi,st) - s^mb(x,\xi,t) & = \underbrace{ [a(x,\delta_s \xi) - s^m a(x, \xi)] \chi_0(st) }_{\in ~ \mathcal{S}_G(U \times \R^{d+1})} \\
& + s^m \Big( \underbrace{a(x,\xi)[\chi_0(st)- \chi_0(t)]  + a(x,\delta_{|t|^{-1}} \xi) |t|^m [\chi_1(st) - \chi_1(t)]}_{:= \tilde{w}} \Big).  \label{37}
\end{align}

Therefore, it remains to prove that $\tilde{w} \in  \mathcal{S}_G (U \times \R^{d+1})$. Using $\chi_1=1-\chi_0$, we have:

\begin{equation} \label{39}
s^m \tilde{w}(x,\xi,t)= s^m \Big( \chi_0(st) - \chi_0(t) \Big) \Big(a(x,\xi) - |t|^m a(x,\delta_{|t|^{-1}} \xi) \Big).
\end{equation}
Then we apply the decomposition from Theorem \ref{1} to the function $a$ and for the compact $K= \overline{B(0,R)}$, the closed ball of radius $R>0$ for the homogeneous quasi-norm:

\begin{equation} \label{154}
\forall ~ x \in U,~  \forall ~ \xi \in  K^c, ~~ a(x,\xi)=a'(x,\xi)+a''(x,\xi),
\end{equation}

where $a' \in \mathcal{H}_G^m(U \times \R^d)$ and  $a'' \in \mathcal{S}_G(U \times \R^d)$.  

We cannot apply this previous decomposition directly to the function $(x,\xi) \mapsto a(x,\delta_{|t|^{-1}} \xi)$ in \eqref{39}, with the same compact $K$ for all $t$. This motivates the following construction.

Thanks to $\chi_0$ in \eqref{39}, note that for fixed $s \in [1,2]$, there exists $0<M_1 \leq 1<M_2$  such that $\tilde{w}(x,\xi,t)=0$, for all $x \in U, \xi \in \R^d$ and for all 
$t \in \Big( [-M_2,-M_1] \cup [M_1,M_2] \Big)^c$.
Taking the compact $\tilde{K}=\overline{B_0(0,\tilde{R})}$, where $\tilde{R} > M_2R>R$ we have:

\begin{equation} \label{152}
\forall ~  \xi \in \tilde{K}^c, ~ \forall ~t \in [-M_2,-M_1] \cup [M_1,M_2], ~~ \delta_{t^{-1}} \xi \in K^c.
\end{equation}

In particular, it is clear that $K \subset \tilde{K}$ and so decomposition \eqref{154} remains true for $\xi \in  \tilde{K}^c$.

Then in equation \eqref{39} we write, for all $x \in U$, for all $\xi \in \R^d$, for all $t \in \R$ :

\begin{align} 
s^m \chi_{\tilde{K}^c}(\xi) & \tilde{w}(x,\xi,t) =s^m[\chi_0(st) - \chi_0(t)] ~ \chi_{\tilde{K}^c}(\xi) \hspace{-5mm} \underbrace{(a'(x,\xi) - |t|^m a'(x,\delta_{|t|^{-1}} \xi))}_{=0, ~ because ~ a' ~ homogeneous ~ on ~ the ~ nose } \hspace{-5mm} , \nonumber \\
\label{40}
& +\chi_{\tilde{K}^c}(\xi) \underbrace{s^m  [\chi_0(st) - \chi_0(t)] a''(x,\xi)}_{ \in ~ \mathcal{S}_G(U \times \R^{d+1}) } + \chi_{\tilde{K}^c}(\xi)(s|t|)^m [\chi_0(st) - \chi_0(t)] a''(x,\delta_{|t|^{-1}}\xi),
\end{align}
where  $\chi_{\tilde{K}^c} \in C_c^\infty(\R^d)$ is a smooth cut-off function, equal to $1$ on $\tilde{K}^c$ and $0$ on $K$. Using proposition \ref{20} we see that $(x,\xi,t) \mapsto (s|t|)^m [\chi_0(st) - \chi_0(t)] a''(x,\delta_{|t|^{-1}}\xi) \in \mathcal{S}_G(U \times \R^{d+1})$.  
Therefore, with equation \eqref{40}, this proves $\chi_{\tilde{K}^c} s^m \tilde{w} \in \mathcal{S}_G (U \times \R^d)$.
Since $(x,\xi,t) \mapsto (1-\chi_{\tilde{K}}(w)) \tilde{w}(x,w,t)$ has compact support in $(\xi,t)$, the reader can easily use equation \eqref{37} to end the proof.
\end{proof}

\subsection{Proof of $2) \Rightarrow 1)$ of Theorem \ref{2}.}
\label{subsection:endproof}
We now take $u \in \mathcal{H} \mathcal{S}_G^m(U \times \R^{d+1})$ from Theorem \ref{2}. 
Also recall that we put $a(x,\xi)=u(x,\xi,1)$. We already showed in Proposition \ref{18} that $a$ is a  symbol class function.  Our goal is now to show that it is polyhomogeneous.
We begin with a straightforward lemma whose proof is left to the reader.

\begin{lemma} \label{22}
We define the subspace $I_0^m:=\lbrace f \in \mathcal{H} \mathcal{S}_G^m(U \times \R^{d+1}) ~ | ~ f(x,\xi,0)=0 \rbrace$ of $\mathcal{H} \mathcal{S}_G^m(U \times \R^{d+1})$. The following map is a linear isomorphism of (Fréchet) spaces: 
$$M_t : \mathcal{H} \mathcal{S}_G^{m-1}(U \times \R^{d+1}) \rightarrow I_0^m, ~ M_t(f)(x,\xi)= tf(x,\xi,t).$$
\end{lemma}
Now we can obtain the asymptotic expansion of $a$ by induction.
\begin{proposition} \label{27} \m
Let $u \in  \mathcal{H} \mathcal{S}_G^m(U \times \R^{d+1})$ and put $a(x,\xi)=u(x,\xi,1)$.
We can construct functions:
\begin{equation} \label{94}
u_j \in \mathcal{H}\mathcal{S}_G^{m-j}(U \times \R^{d+1}),
\end{equation}

with the property that if we set $a_j(x,\xi)=u_j(x,\xi,0) \in \mathcal{H}\mathcal{S}_G^{m-j}(U \times \R^{d})$, then:

\begin{equation} \label{95}
\forall ~ j \in \mathbb{N}^*, u_j(x,\xi,1)=a(x,\xi)-\sum_{i=0}^{j-1} a_i(x,\xi).
\end{equation}

\end{proposition}

\begin{proof}
For $j=0$, we put $u_0=u$ and the equation \eqref{95} is trivial.
Suppose that we have constructed the desired functions $u_0,...,u_j$. Then we define the function $b_j$ as in equation \eqref{23}:

\begin{equation} \label{71}
b_j(x,\xi,t):=\chi_0(t) a_j(x,\xi)+ \chi_1(t)|t|^m a_j(x,\delta_{|t|^{-1}} \xi).
\end{equation}

As $a_j \in \mathcal{H} \mathcal{S}_G^{m-j}(U \times \R^d)$, Theorem \ref{29} allows us to claim that $b_j \in \mathcal{H}\mathcal{S}_G^{m-j}(U \times \R^{d+1})$.
Note that $u_j(x,\xi,0)-b_j(x,\xi,0)=0$. Thanks to the isomorphism of Lemma \ref{22} and the induction hypothesis,  there exists $u_{j+1}\in \mathcal{H}\mathcal{S}_G^{m-j-1}(U \times \R^{d+1})$ satisfying $tu_{j+1}=u_j - b_j$ and so it is legitimate to write the following equation:

\begin{equation} \label{21}
u_{j+1}(x,\xi,t)=\frac{u_j(x,\xi,t) - b_j(x,\xi,t)}{t} \in \mathcal{H} \mathcal{S}_G^{m-j-1}(U \times \R^{d+1}).
\end{equation}

We now set, $a_{j+1}:=u_{j+1}|_{t=0} $ and it is an element of $\in \mathcal{H}\mathcal{S}_G^{m-j-1}(U \times \R^{d})$.
It remains to show equality \eqref{95}. This follows from \eqref{21} and the inductive hypothesis, after observing:
\begin{equation}
 b_j(x,\xi,1)=a_j(x,\xi).
\end{equation}

Now the induction is complete and so the proposition is done.
\end{proof}


\begin{proof}[Proof of $2) \Rightarrow 1)$ of Theorem \ref{2}.]
First of all, according to Proposition \ref{18} point 2), $a=u|_{t=1} \in S_G^m(U \times \R^d)$. We now use the Proposition \ref{27} to construct the asymtotic expansion of $a$, consult Definition \ref{12}.
Thanks to Proposition \ref{27},  we can write for every natural number $N$:

\begin{equation} \label{70}
a(x,\xi)=\sum_{j=0}^{N} a_j(x,\xi) + u_{N+1}(x,\xi,1),
\end{equation}

where $a_j \in \mathcal{H} \mathcal{S}_G^{m-j}(U \times \R^d)$ and $u_{N+1} \in \mathcal{H} \mathcal{S}_G^{m-N-1}(U \times \R^{d+1})$.
Again, thanks to Proposition \ref{18} point 2), $u_{N+1}|_{t=1} \in S_G^{m-N-1}(U \times \R^d)$.
Finally, with a slight modification of the way we write the conclusion of Theorem \ref{1}, we can write for every $j$:

\begin{equation} \label{50}
a_j(x,\xi)=\chi_K(\xi)a'_j(x,\xi)+a_j''(x,\xi),
\end{equation}
 
where for every natural number $j$, $a'_j \in \mathcal{H}_G^{m-j}(U \times \R^d)$, $ a_j'' \in \mathcal{S}_G(U \times \R^d)$ and $\chi_K$ is any cut-off function as in Definition \ref{12}. 
Therefore according to Definition \ref{12}, $\sum_{j} a_j'$ is an asymptotic expansion for $a$, 
\end{proof}

\section{Extending polyhomogeneous functions.}
\label{sec:section4}

\subsection{Construction of a good candidate.}
\label{subsec:backtopreli}

Given $a \in S_{phg,G}^m(U \times  \R^d)$ we will now build an element $u \in \mathcal{H} \mathcal{S}_G^m(U \times \R^{d+1})$ such that $a$ is the restriction at time $t=1$ of $u$.
We therefore suppose we have an asymptotic expansion $a \sim \sum_j a_j$ in the sense of Definition \ref{12}.




We set, with $|.|$ being the homogeneous quasi-norm, the cut-off function $\phi \in C^\infty(\R^d)$:
\begin{equation} \label{56}
\phi(\xi) = \left\{
    \begin{array}{ll}
        1 & \mbox{if } |\xi| \geq 1 \\
        0 & \mbox{if } |\xi| \leq  \frac{1}{2} 
    \end{array}
\right.
\end{equation}

The following result is a generalisation of a well-known result, see  \cite[Theorem 3.1 p 40]{Taylor1982PsiDo}  and  \cite[Theorem 6.10 p 36]{Wong2014PsiDO} with more details. 

\begin{proposition} \label{57}
There exists a sequence $(\epsilon_j)$ converging to $0$ so that the series 
$(x,\xi) \mapsto \sigma(x,\xi):=\sum_{j=0}^{+ \infty} \phi(\delta_{\epsilon_j} \xi) a_j(x,\xi)$ belongs to  $S_G^m(U \times \R^d)$.
\end{proposition}

Motivated by this result, we set:

\begin{equation} \label{58}
a_j'(x,\xi):=\phi(\delta_{\epsilon_j} \xi) a_j(x,\xi),
\end{equation}

and readily obtain $a_j' \in \mathcal{H} \mathcal{S}_G^{m-j}(U \times \R^d).$
We also introduce the following partition of unity of  $\R$: 

 \begin{equation} 
\chi_0(t) = \left\{
    \begin{array}{ll}
        0 & \mbox{si } |t| \geq 1 \\
        1 & \mbox{si } |t| \leq \frac{1}{2}
    \end{array}
\right.
\end{equation}

\begin{equation}
\chi_1=1-\chi_0.
\end{equation}

We set:

\begin{equation} \label{60}
b_j(x,\xi,t):=\chi_0(t) a_j'(x,\xi)+\chi_1(t) |t|^{m-j} a_j'(x,\delta_{|t|^{-1}} \xi).
\end{equation}

Note that $b_j$ is constructed on the same model as $b_0$ from equation \eqref{23}.
Therefore, thanks to Theorem \ref{29}, $b_j \in \mathcal{H} \mathcal{S}_G^{m-j}(U \times \R^{d+1})$. Finally we set:

\begin{equation} \label{61}
b(x,\xi,t):=\sum_{j=0}^{+ \infty} t^j b_j(x,\xi,t).
\end{equation}

The key point in what follows is to show $b \in \mathcal{H} \mathcal{S}_G^m(U \times \R^{d+1})$: this is exactly Theorem \ref{68}.
Note that the series $b$ defined at equation \eqref{61} is locally finite and therefore is $C^\infty$ on $U \times \R^{d+1}$.

Recall from the remark after Definition \ref{10} that it suffices to consider $s \in [1,2]$.
We have, for all $ s ~ \in [1,2]$, for all $ (\xi,t) \in \R^{d+1}$, for all $ x \in U$:

\begin{align} 
 b(x,\delta_s \xi,st)- s^m b(x,\xi,t) & =\sum_{j=0}^{+ \infty} (st)^j b_j(x, \delta_s \xi,st) -s^m \sum_{j=0}^{+ \infty} t^jb_j(x,\xi,t) \\
&= \sum_{j=0}^{+ \infty} (st)^j f_{j,s}(x,\xi,t),  \label{64}
\end{align}
where we set:

\begin{equation} \label{65}
f_{j,s}(x,\xi,t):=b_j(x,\delta_s \xi,st) - s^{m-j} b_j(x,\xi,t).
\end{equation}
From equation \eqref{65}, the reader can observe that the sequence of functions $(f_{j,s})$ is smooth with compact support in:
\begin{equation} \label{90}
 U \times \overline{B(0,\epsilon_j^{-1}}) \times [-1,1], ~  \text{for fixed} ~ s \in [1,2],	
\end{equation}
where $\overline{B(0,\epsilon_j^{-1}})$ is the closed ball for the homogeneous quasi-norm, since $a_j'$ is homogeneous of order $m-j$ on the nose outside this set.
 
Unfortunately, as one can see in equation \eqref{90},  as $j$ gets bigger, the sequence $(\epsilon_j^{-1})$ goes to $+ \infty$, meaning that the support of the functions $f_{j,s}$ will spread to infinity. Nonetheless, if we redefine the sequence $(\epsilon_j)$ -- reducing each term so that Proposition \ref{57} is still true -- we will get rapid convergence in the series \eqref{64}.
Therefore, the vital cog of this part is to redefine a sufficiently rapidly converging sequence $(\epsilon_j)$. This is what we do in equations \eqref{76} and \eqref{77} below, to obtain the Theorem \ref{68}.

The following simple lemma will be useful:
\begin{lemma}\label{67}
Given $F \in C^{\infty}(U \times \R^{d+1})$. We suppose the following conditions to be true :
\begin{enumerate}
\item{$supp(F) \subset U \times \R^d \times [-1,1]$.}
\item{For every  $(\beta,\gamma)$ multi-indices, $\alpha \in \N$ for every compact K included in $U$, $\exists ~  C_{\beta,\alpha,\gamma,K} >0 $ so that $\forall ~ x,\xi \in K \times \R^d$, $\forall ~ t \in \R$ the following holds: 
$$|\partial_x^\gamma \partial_{\xi}^{\beta} \partial_t^{\alpha} F(x,\xi,t)| \leq C_{\beta,\alpha,\gamma,K}  (1+|\xi|)^{m-|\beta|- \alpha},$$}
\end{enumerate}
then $F \in S_G^m(U \times \R^{d+1})$.
\end{lemma} 
The proof is not hard and is left to the reader.



\begin{theoreme}\label{68}
We have  $b \in \mathcal{H} \mathcal{S}_G^m(U \times \R^{d+1})$.
\end{theoreme}


\begin{proof}

To show that $b \in \mathcal{H} \mathcal{S}_G^m(U \times \R^{d+1})$, we benefit from the equalities:

\begin{equation} \label{69}
S_G^{- \infty}(U \times \R^{d+1})=\cap_{ N \in \mathbb{N}} S_G^{m -N}(U \times \R^{d+1})=\mathcal{S}_G(U \times \R^{d+1}),
\end{equation}

and show: 

\begin{equation} \label{155}
\forall ~ N \in \mathbb{N}, \forall ~ s \in [1,2], ~~  (x,\xi,t) \mapsto \sum_{j=0}^{+ \infty} (st)^j f_{j,s}(x,\xi,t) \in S_G^{m -N}(U \times \R^{d+1}).
\end{equation}

As usual, we fix $s \in [1,2]$. Let $N \in \N$ and write:

\begin{align} \label{72}
\sum_{j=0}^{+ \infty} (st)^j f_{j,s}(x,\xi,t)=\sum_{j=0}^{N-1} (st)^j f_{j,s}(x,\xi,t) \underbrace{+ \sum_{j=N}^{+ \infty} (st)^j a_j'(x,\delta_s \xi)  - s^m \sum_{j=N}^{+ \infty}  t^ja_j'(x,\xi).}_{(x,\xi,t) \mapsto \sum_{j=N}^{+ \infty} (st)^j f_{j,s}(x,\xi,t)  }
\end{align}

It is easy to see that the first term on the right of the previous equality is smooth and compactly supported. Therefore, thanks to \eqref{72}, it suffices to show that the two infinite series satisfy the conditions of lemma \ref{67}. We will do this for the first of the two, the other being similar. The functions $a_j$ are in $C^\infty(U \times \R^d \setminus \{ 0 \})$, homogeneous of order $m-j$ on the nose, by hypothesis. So thanks to proposition \ref{92},  $a_j'=((\delta_{\epsilon_j})^* \phi) a_j \in  S_G^{m-j}(U \times \R^d)$, where the upper star denotes the pull-back. 
\bg
We fix an exhaustive sequence of compacts $(K_i)$ covering $U$.
We find that, for any natural number $j$, 
for any compact set $K$ of $U$,  $\forall ~ \beta,\gamma$ multi-indices, there exists a natural number $i$ such that $K \subset K_i$ and $ C_{\beta,\gamma,j,K_i} >0$ (independant of $s$) such that $\forall ~ x,\xi \in K \times \R^d$:

\begin{align} \label{75}
 |\partial_x^{\gamma} \partial_{\xi}^{\beta} \phi(\delta_{\epsilon_j s} \xi) a_j(x,\delta_s \xi)| \leq C_{\beta,\gamma,j,K_i} (1 + |\xi|)^{m - j - |\beta |}.
\end{align}

We will replace the sequence $(\epsilon_j)$ by a sequence $(\epsilon_j')$, with  $0 \leq \epsilon_j' \leq \epsilon_j$, for all $j$. Thus $(\epsilon_j')$ still converges to 0. We also demand that $(\epsilon_j')$ satisfyies:

\begin{align}\label{76}
\epsilon_j' & \leq 2^{-2j} ~ \min \left\{ \left. \frac{1}{C_{\beta,\gamma,j,K_i}} ~ \right| ~~ \gamma,\beta,i ~\text{such that}~ |\beta |+|\gamma|+i \leq j \right\}, \\
\epsilon_j' & \leq \frac{1}{4}. \label{77}
\end{align}

Recall equation \eqref{56} defining $\phi$. It follows that:
\begin{itemize}
\item{If $\xi$ is such that $1 + |\xi| \leq \frac{1}{2} (s\epsilon_j')^{-1}$ then in particular $|\xi| \leq \frac{1}{2} (s\epsilon_j')^{-1} $ so $\phi(\delta_{\epsilon_j' s} \xi) =0$.}
\item{If $\xi$ is such that  $1 + |\xi| \geq \frac{1}{2} (s\epsilon_j')^{-1}$ then $(1+|\xi|)^{-1} \leq 2 s \epsilon_j'$.}
\end{itemize}

The right member of equation \eqref{75} can be rewriten as:
\begin{equation}\label{78}
 |\partial_x^{\gamma}  \partial_{\xi}^{\beta} \phi(\delta_{\epsilon_j' s} \xi) a_j(x,\delta_s \xi)| \leq C_{\beta,\gamma',j,K_i} (1 + |\xi|)^{m - j +1 - |\beta |}(1 + |\xi|)^{-1}.
\end{equation}
According to \eqref{76}, \eqref{77} and the remark following them, the equation \eqref{78} can be rewriten as: 
%
\begin{equation}\label{79} 
| \partial_x^{\gamma}  \partial_{\xi}^{\beta} \phi(\delta_{\epsilon_j' s} \xi) a_j(x,\delta_s \xi)| \leq  4 ~ 2^{-2j} ~(1 + |\xi|)^{m - j +1 - |\beta |}.
\end{equation}
From now fix $i$, $\beta_0,\gamma_0$ multi-indices and $\alpha_0 \in \N$.  At this juncture of the proof, we take $j_0$ big enough, depending on $N$,  $\beta_0,\gamma_0$, $\alpha_0$ and $i$ so that:
\begin{align}
j_0 & \geq  |\beta_0| + |\gamma_0| + i, \\
 j_0  & \geq  N+1+ \alpha_0.  \label{81}
\end{align}
Because the following series are locally finite, we can write:
\begin{align}
| \partial_x^{\gamma_0} \partial_t^{\alpha_0} \partial_{\xi}^{\beta_0} & \sum_{j=N}^{+ \infty}  (st)^j \phi(\delta_{\epsilon_j' s} \xi) a_j(x,\delta_s \xi)| \nonumber \\
& = | \partial_x^{\gamma_0} \partial_t^{\alpha_0} \partial_{\xi}^{\beta_0} \sum_{j=N}^{j_0 -1} (st)^j \phi(\delta_{\epsilon_j' s} \xi) a_j(x,\delta_s \xi) + \sum_{j=j_0}^{+ \infty} \partial_x^{\gamma_0} \partial_t^{\alpha_0} \partial_{\xi}^{\beta_0} (st)^j \phi(\delta_{\epsilon_j' s} \xi) a_j(x,\delta_s \xi)| \label{82} 
\end{align}



Again, the finite sum can be ignored.

In the infinite sums we have $j \geq j_0 \geq |\beta_0| + |\gamma_0| + i $, and in particular  $j \geq |\beta_0|$. So according to equation \eqref{79} we have $\forall ~(x,\xi,t) \in K \times \R^d \times [-1,1]$:

\begin{align}
|\sum_{j=j_0}^{+ \infty} \partial_x^{\gamma_0} \partial_t^{\alpha_0} \partial_{\xi}^{\beta_0} (st)^j \phi(\delta_{\epsilon_j' s} \xi) a_j(x,\delta_s \xi)|
& \underbrace{\leq}_{s \in [1,2]}   \sum_{j=j_0}^{+ \infty } 2^j \frac{j!}{(j- \alpha_0)!} | \partial_x^{\gamma_0} \partial_{\xi}^{\beta_0} \phi(\delta_{\epsilon_j' s} \xi) a_j(x,\delta_s \xi)| \nonumber \\ 
& \underbrace{\leq}_{\eqref{79} } 4 ~ \sum_{j=j_0}^{+ \infty } 2^{-j} \frac{j!}{(j- \alpha_0)!} (1+|\xi|)^{m-j+1-|\beta_0|}  \nonumber \\
& \leq 4 ~ (1+|\xi|)^{m-N-|\beta_0|-\alpha_0} ~ \sum_{j=j_0}^{+ \infty } 2^{-j} \frac{j!}{(j- \alpha_0)!}  \underbrace{(1+|\xi|)^{N -j +1 +\alpha_0}}_{ \leq  1 }, \label{85}
\end{align}

where the ineqality in the last underbrace is true because we have at the same time $j \geq j_0$ and $j_0 \geq N+1+\alpha_0$, by equation \eqref{81}.
D'Alembert's rule shows that the series in \eqref{85} is a convergent series. Therefore with equation \eqref{85} it follows that:

\begin{equation}\label{87}
| \partial_x^{\gamma_0} \partial_t^{\alpha_0} \partial_{\xi}^{\beta_0} \sum_{j=j_0}^{+ \infty}  (st)^j \underbrace{\phi(\delta_{\epsilon_j' s} \xi) a_j(x,\delta_s \xi)}_{:=a_j'(x,\delta_s \xi)}|  \leq C_{N,\beta_0,\alpha_0,\gamma_0,K} (1+|\xi|)^{m - N - |\beta_0| - \alpha_0},
\end{equation}

where we set $C_{N,\beta_0,\alpha_0,\gamma_0,K}:= 4 \sum_{j=j_0}^{+ \infty}  2^{-j} \frac{j!}{(j- \alpha_0)!}  > 0$.

This shows that the first infinite series in equation \eqref{72}  satisfies lemma \ref{67}. The second series  is similar.
With this information in mind, equation \eqref{72} and the remark following it shows that for all $N$, the series  $(x,\x,t) \mapsto \sum_{j=j_0}^{+ \infty} (st)^j f_{j,s}(x,\xi,t) \in S_{G}^{m-N}(U \times \R^{d+1})$. The result follows thanks to the discussion at the beginning of the proof.
\end{proof}


We can now conclude the proof of Theorem \ref{2}:

\begin{proof}[Proof $1) \Rightarrow 2)$ of Theorem \ref{2}.]
Let $a \in S_{phg,G}^m(U \times \R^d)$.
 Theorem \ref{68} tells us that $b \in  \mathcal{H}\mathcal{S}_G^{m}(U \times \R^{d+1})$. Moreover we have $b(x,\xi,1)=\sum_{j=0}^{+ \infty} a_j'(x,\xi)= \sum_{j=0}^{+ \infty} \phi(\delta_{\epsilon_j} \xi) a_j(x,\xi)$.
Hence, $a$ and $b$ have the same asymptotic expansion and so differ by a function in  $S_G^{- \infty}(U \times \R^d)=\mathcal{S}_G(U \times \R^d)$. Thus  $l:=a-b(.,.,1) \in \mathcal{S}_G(U \times \R^d)$. 
Now put $u(x,\xi,t)=b(x,\xi,t)+l(x,\xi) \tilde{\phi}(t)$, where $\tilde{\phi} \in C_c^\infty(\R)$ with $\tilde{\phi}=1$ on $[-1,1]$. Then $u(x,\xi,1)=a(x,\xi)$. Since $(x,\xi,t) \mapsto l(x,\xi)\tilde{\phi}(t)$ is Schwartz class, it follows that $u \in \mathcal{H} \mathcal{S}^m(U \times \R^{d+1})$. This completes the proof.

\end{proof}

\section{Application to Heisenberg calculus on contact manifolds.}
\label{sec:section5}

\subsection{Heisenberg manifold of type $\Hn \times \R^m$ and its exponential charts.}

In this section we investigate the Heisenberg calculus. Our goal is to compare the calculus on a Heisenberg manifold from Beals and Greiner \cite{Beals2016Heisenbergcalculus} and the groupoidal pseudodifferential calculus  of Van Erp and the second author \cite{Yuncken2019groupoidapproach}, developed in the case of filtered manifolds. We will show, see Theorems \ref{157} and \ref{foliations}  that the two theories coincide in the case of contact manifolds and foliations. From now on, we refer to these authors and their works, respectively as BG and vEY. The reader can find a nice exposition of Heisenberg manifolds in \cite{ponge2008heisenberg}, and in particular section 2.1 which illustrates the fact that these manifolds are very common.
The next definition is taken from BG \cite{Beals2016Heisenbergcalculus} Definition (10.1) p90.
\begin{definition} \label{97}
A Heisenberg manifold $M$ is a d+1 dimensional $C^\infty$ manifold equipped with a hyperplane bundle $\V$, that is, a subbundle $\V \subset TM$ such that for each $y \in M$, the fiber $\V_y$ has codimension one in the tangent space $T_yM$.
\end{definition} 

\textbf{Remark:} In the language of vEY, following the preprint of Melin \cite{melin1982lie} from 1982 Definition 2.1,  a Heisenberg manifold is a two-step  filtered manifold with the following filtration:

\begin{equation}
H^0=\{0\} \leq H^1=\V \leq H^2=TM.
\end{equation}

In the first few sections, we will deal only with the model situation.
We introduce the following vector fields on $\R^{d+1}$:

\begin{equation} \label{101}
\begin{cases}
    X_0=\frac{\partial}{\partial x_0}  \\
     X_j=\frac{\partial}{\partial x_j} + \frac{1}{2} \sum_{k=1}^d b_{jk} x_k \frac{\partial}{\partial x_0}, &   \forall ~ 1 \leq j \leq d  ,
\end{cases}
\end{equation}

where $B=(b_{jk})$ is an anti-symmetric matrix of order $d$.
Moreover we let $\mathcal{V}$ be the hyperplane bundle on  $M=\R^{d+1}$ generated  by $X_1,...,X_d$.
The vector fields $(X_j)$ are left-invariant for the following group law on $\R^{d+1}$:

\begin{equation} \label{99}
\begin{cases}
    (x.x')_0=x_0+x_0'+\frac{1}{2} \sum_{k,j=1}^d b_{jk} x_k x_k'  \\
     (x.x')_j=x_j+x_j', & ~ \forall ~ j \in \{1,...,d \},
\end{cases}
\end{equation}

where $x_j$ denotes the $j$ th coordinate of $x \in \R^{d+1}$.

\begin{exemple}
The case $d=2n$ and 
\begin{equation}
B= \begin{pmatrix}
0_n & -I_n \\
I_n & 0_n \\
\end{pmatrix},
\end{equation}
corresponds to the Heisenberg group. 
The case where $ b_{jk}=0, \forall ~ (j,k)$ yields the abelian group $\R^{d+1}$.
\end{exemple}







From now on, $M$ stands for $\R^{d+1}$ equipped with the above group law \eqref{99}. Then $M \cong \Hn \times \R^m$, for some $n,m$ with $2n+m=d$. The manifold $M$ is a filtered manifold in the sense described in \cite{melin1982lie}, \cite{choi2019tangent} or \cite{Yuncken2017tangentgroupoidfilteredmanifold}, \cite{Yuncken2019groupoidapproach}. In this particular case, we can equip this manifold with dilations and exponential charts as follows.

\begin{definition} \label{103}
We define the following family of dilations, called the Heisenberg dilations. For $s \in \R_+^*$:
\begin{equation} \label{147}
\delta_s : (v_0,v_1,...,v_d) \in \R^{d+1} \mapsto (s^2 v_0,sv_1,...,sv_d) \in \R^{d+1}.
\end{equation}

\end{definition}
\textbf{Remark:} This is a dilation as in \eqref{3}. One can also encounter this dilation in the original paper \cite{folland1974estimatesHeisenberg} of Folland and Stein, see section 6 p439. 
Dilations are associated to a notion of order for vector fields. In the case of \eqref{147}, the vector field $X_0$ has order 2 and the vector fields $(X_j)_{j \neq 0}$ have order 1.
\begin{exemple}
In the Heisenberg group $\mathbb{H}_1$, with vector fields $(X_0,X_1,X_2)$ such that $[X_1,X_2]=X_0$, the differential operator $P=X_1^2+X_2^2+X_0$ is homogeneous of order 2 in the sense of dilatation \ref{103}. Note that this would not be the case if $\mathbb{H}_1$ had been trivialy filtered.
\end{exemple}
\textbf{Some notation :} Given a family $\overline{X}=(X_1,...X_d)$  of vector fields on a smooth manifold $M$ and a vector $v \in \R^{d+1}$, we write $v.\overline{X}=\sum_{k=1}^d v_k X_k$. We denote $\exp(X).x$ to be the flow of the point $x \in M$ at time 1 along the vector field $X$.

At each point $y$ in a Heisenberg manifold, one can define an osculating group $\mathcal{T}_HM_y$, see \cite{Beals2016Heisenbergcalculus}. In the model case we are currently considering, we have a canonical identification $\mathcal{T}_HM_y \cong \Hn \times \R^n$. The Heisenberg tangent groupoid, generalising Alain Connes' famous tangent groupoid $\TM$, is denoted $\THM$. Recall briefly that in this case, see \cite{Yuncken2017tangentgroupoidfilteredmanifold}:

\begin{equation} \label{104}
\THM=\Big( \underbrace{M \times M}_{\text{pair groupoid}}  \times ~~ \R^* \Big) \bigcup  \Big( \hspace{-5mm} \underbrace{\mathcal{T}_HM}_{\text{osculating  groupoid}} \hspace{-5mm} \times ~~ \{ 0 \} \Big).
\end{equation}
We remind the reader that $\mathcal{T}_HM$ is a bundle of nilpotent Lie groups whose fibers are precisely the $\mathcal{T}_HM_y$. The groupoid $\THM$ is equipped with the following structure.
\begin{definition} \label{105}
With the fields defined at equation \eqref{101}, we define the following chart on the Heisenberg manifold $M$, called an expoential chart:

\begin{equation} \label{106}
\Exp : M \times \R^{d+1} \times \R \rightarrow \THM,  (y,v,t) \mapsto 
\begin{cases}
(y,\exp(\delta_t(-v).\overline{X}).y,t) & \mbox{if $t \neq 0$} \\
(y,v.\overline{X}|_y,0) & \mbox{if $t=0$.}
\end{cases}
\end{equation}

\end{definition}
These charts, for different choices of $\overline{X}$ a local frame compatible with the filtration, do indeed define a smooth structure on $\THM$.
 In the case of the vector fields $(X_0,...X_d)$ from equation \eqref{101}  we can compute the chart to obtain an explicit formula:

\begin{proposition} \label{140}
The charts $\Exp$ are given at $t \neq 0$ by:
\begin{equation} \label{107}
\Exp(y,v,t)=\Big(y,y',t \Big),
\end{equation}
where $y'= \Big(  - \frac{1}{2} \sum_{j=1}^d \sum_{k \neq j} v_j b_{jk}y_k t -v_0t^2 + y_0 , -v_1t +y_1, \ldots , -v_dt +y_d  \Big)$. In particular, 

\begin{equation} \label{143}
\Exp(y,v,1)=(y,\phi_y(v)+y,1),
\end{equation}

where we define the bijective linear  map $\phi_y(v)=\Big(-v_0 - \frac{1}{2} \sum_{j=1}^d \sum_{k \neq j} v_j b_{jk} y_k, -v_1, \ldots , -v_d \Big)$.
\end{proposition}
\begin{proof}
One can check that the flow of $y=(y_j)$ along the vector field $v.\overline{X}$ is given by:

\begin{align}
\gamma(s) &=\Big( y_0+v_0s+\frac{1}{2} \sum_{j,k=1}^d v_j b_{jk} (y_k s+v_k \frac{s^2}{2}) ),y_1+v_1s,...,y_d+v_d s \Big) \\
& = \Big( y_0+v_0s+\frac{1}{2} \sum_{j=1}^d \sum_{k \neq j} v_j b_{jk} y_k s, y_1+v_1s,...,y_d+v_d s \Big),
\end{align}

where in the second equality we have used the antisymmetry of $(b_{jk})$. Putting $s=1$ and replacing $v$ by $\delta_t(-v)=(-t^2v_0,-tv_1,...,-tv_d)$, we obtain the exponential map $\Exp$ from \eqref{106}. The result follows. 
\end{proof}

\subsection{The Heisenberg pseudo-differential calculus according to BG.}
In this subsection we recall the definition of the calculus of BG and show how the Fourier transform connects the symbols of their calculus with the exponential charts $\Exp$ above. 
The following properties can be found on page 26 of BG \cite{Beals2016Heisenbergcalculus} and we follow their notation.
The operators defined by the vector fields of equation \eqref{101} have symbols:

\begin{equation} \label{110}
\left\{
    \begin{array}{ll}
      \sigma_0(x,\eta)=\eta_0  \\
    \sigma_j(x,\eta)=\eta_j + \frac{1}{2} \sum_{k=1}^d b_{jk} x_k \eta_0,  ~~ \forall ~ 1 \leq j \leq d.  
    \end{array}
\right.
\end{equation}

Let us write $\sigma(x,\xi)=(\sigma_0(x,\xi),...,\sigma_d(x,\xi))$ and $\overline{\sigma}(x,\xi)=(x,\sigma(x,\xi))$.
We can define the inverse $\widetilde{\sigma}$ of $\sigma$, in the sense that $\sigma(x,\widetilde{\sigma}(x,\eta))=\eta$, by: 

\begin{equation} \label{111}
\left\{
    \begin{array}{ll}
      \widetilde{\sigma_0}(x,\eta)=\eta_0  \\
   \widetilde{\sigma_j}(x,\eta)=\eta_j - \frac{1}{2} \sum_{k=1}^d b_{jk} x_k \eta_0,  ~~ \forall ~ 1 \leq j \leq d.  
    \end{array}
\right.
\end{equation} 
Thus $\overline{\sigma}^{-1}(y,v)=(y,\widetilde{\sigma}(y,v))$.
The following definition of $\V$-pseudodifferential operator can be found page 91 of BG \cite{Beals2016Heisenbergcalculus}. The equation \eqref{113} that defines the pseudodifferential operators is extracted from Definition (9.1) p80 of \cite{Beals2016Heisenbergcalculus} and the expression of the kernel is given at equation (9.12) p81 of \cite{Beals2016Heisenbergcalculus}.
\begin{definition} \label{112}
Let $U \subset \R^{d+1}$ be an open set.
We define $S_{\V}^m(U \times \R^{d+1})$ to be the $\V$-symbol class. It consists of functions $q \in C^\infty(U \times \R^{d+1})$, which can be put in the following form 
\begin{equation}
q(x,\xi)=f(x,\sigma(x,\xi))=\overline{\sigma}^* f(x,\xi),
\end{equation}
for some $f \in S_{phg,G}^m(U \times \R^{d+1})$, see Definition \ref{12}, where we use the Heisenberg dilations of Definition \ref{103}.
Given $q \in S_{\V}^m(U \times \R^{d+1})$ we define, for $\phi \in C_c^\infty(\R^{d+1})$, the operator 

\begin{equation} \label{113}
Op(q)\phi(x)=\int_{\R^{d+1}} e^{i x .\xi} q(x,\xi) \hat{\phi}(\xi) \frac{d \xi}{(2 \pi)^{d+1}},
\end{equation}
and call it a $\V$-pseudodifferential operator with symbol $q$. Its Schwartz kernel $k \in D'(U \times U)$ is  the distribution 

\begin{equation} \label{114}
k(x,y)=\mathcal{F}_2^{-1}(q)(x,x-y) \in D'(U \times U),
\end{equation}
where $\mathcal{F}_2$ denotes the Fourier transform in the second variable, that is $\mathcal{F}_2^{-1}(q)(x,\xi)=\mathcal{F}^{-1}(v \mapsto q(x,v))(\xi)$.
We denote by $\bold{\Psi}^m_{\mathcal{V}}(U)$ the set of  $\V$-pseudodifferential operators.
\end{definition}
The first result we obtain is a simple link between $f$ and the Fourier transform of the Schwartz kernel of $Op(q)$ pushed forward by the exponential map. This result gives the conceptual explanation for the agreement between the calculus of BG and the groupoidal calculus.
Before that, let us note that the linear part $\phi_y$ from Proposition \ref{140} is related to $\widetilde{\sigma}$ by the following formula which follows directly from equations \eqref{143} and \eqref{111}.
\begin{proposition} \label{116}
The bijective linear  map $\phi_y$ satisfies:
$$^t \phi_y^{-1} (\eta) = \widetilde{\sigma}(y,-\eta).$$
\end{proposition}
\begin{theoreme} \m \label{108}
Let $f \in S_{phg,G}^m(M \times \R^{d+1})$ where $M=\R^{d+1}$ as above. 
We denote by $k \in D'(M \times M)$ the Schwartz kernel of the operator $Op(q)$ from BG, where $q$ is defined in Definition \ref{112}.
With the map $\Exp$ from Definition \ref{105}, we set $\widetilde{k}:=(\Exp |_{t=1})_*^{-1} k$, the pushforward of the distribution $k$ by the exponential chart. Then, 
$$\mathcal{F}_2(\widetilde{k})(y,\xi)=f(y,\xi).$$
\end{theoreme}
\textbf{Remark:}
We can sum up this result with the following commutative diagramme:

$$\xymatrix@C+=2cm{k \ar@{|->}[r]^{ (\Exp |_{t=1})_*^{-1} } \ar@{|->}[d]_{\mathcal{F}_2 \circ (\tau \circ m)_* } & \widetilde{k} \ar@{|->}[d]^{\mathcal{F}_2} \\ q \ar@{|->}[r]_{(\overline{\sigma}^{-1})^*} & f }$$

where $\tau : (y,v) \in  \R^{d+1} \times \R^{d+1} \mapsto (y,y+v)$ and $m : (y,v) \in  \R^{d+1} \times \R^{d+1} \mapsto (y,-v)$. 

The consequence of Theorem \ref{108} will be that we can apply Theorem \ref{2} to $f \in S_{phg}^m(U \times \R^{d+1})$. Therefore, there exists $u \in \mathcal{H} \mathcal{S}^m(U \times \R^{d+1} \times \R)$ such that $u|_{t=1}=\mathcal{F}_2(\widetilde{k})$.

\begin{proof}[\text{Proof of theorem \ref{108}}]
Introduce $\Phi : (x,w) \mapsto (x,\phi_x(w))$ where again, $\phi_x$ is the linear map from Proposition \eqref{140}.
Let us remind the reader that if $\phi : \R^d \rightarrow \R^d$ is a linear map and $\mathcal{F}$ denotes the Fourier transform on $\R^d$, then $\mathcal{F} \circ \phi_*= ~ ^t\phi^* \circ \mathcal{F}$, where $\phi_* : \mathcal{E}'(\R^d) \rightarrow \mathcal{E}'(\R^d)$ is the pushforward of distributions and $\phi^* : C^\infty(\R^d) \rightarrow C^\infty(\R^d)$ the pull-back of functions.
Also note that with Definition \ref{112}, $k(x,y)=\mathcal{F}_2^{-1}(q)(x,x-y)$ gives $k(y,y-v)=\mathcal{F}_2^{-1}(q)(y,v)$. Since $(\tau \circ m)^{-1} = \tau \circ m$ we get:

\begin{equation} \label{160}
\mathcal{F}_2 \Big( (\tau \circ m)_* k \Big)(y,\xi)=q(y,\xi).
\end{equation}


One can compute the following:
\begin{align*}
\mathcal{F}_2 \Big( (\Exp |_{t=1})_*^{-1}(k) \Big)(y,\xi) & \hspace{-7mm} \underbrace{=}_{equation ~ \eqref{143}} \hspace{-4mm} \mathcal{F}_2 \Big( (\tau \circ \Phi)_*^{-1}k \Big)(y,\xi) \\ &= \hspace{+2mm} \mathcal{F}_2 \Big( \Phi_*^{-1} \circ \tau_*^{-1} k \Big)(y,\xi) \\
&  =  \mathcal{F}_2 \Big( \tau_*^{-1} k \Big)(y,^t(\phi_y^{-1})(\xi)) \\ & \hspace{-7mm} \underbrace{=}_{proposition ~ \ref{116} } \hspace{-5mm} \mathcal{F}_2 \Big( \tau_*^{-1} k \Big)(y,\widetilde{\sigma}(y,-\xi)) \\
&=\hspace{5mm} \mathcal{F}_2 \Big( ( \tau \circ m)_* k \Big)(y,\widetilde{\sigma}(y,\xi)) \\
& \underbrace{=}_{\eqref{160}} \hspace{5mm} q(y,\widetilde{\sigma}(y,\xi)) \\
& = f(y,\xi),
\end{align*}

where the last equality is true because $\sigma(y,\widetilde{\sigma}(y,\xi))=\xi$, see just after equation \eqref{111}, and $q(y,\xi)=f(y,\sigma(y,\xi))$ by definition.
\end{proof}
\subsection{Equality of the two calculi.}
We recall the Debord-Skandalis action on $\THM$, see \cite{Yuncken2019groupoidapproach} (where it is called the zoom action) and \cite{debord2014adiabatic}:
\begin{definition} \label{121}
We definine the Debord-Skandalis action of $\R_+^*$ on $\THM$, $s \in \R_+^* \mapsto \alpha_s \in Aut(\THM)$ by:
\begin{equation}
\left\{
    \begin{array}{ll}
     \alpha_s(y,x,t)=(y,x,s^{-1}t) ~~ (x,y) \in M  \\
    \alpha_s(x,\xi,0)=(x,\delta_s(\xi),0) ~~ x \in M, ~ \xi \in\mathcal{T}_HM_x.
    \end{array}
\right.
\end{equation}

\end{definition}
\textbf{Remark:} Each $\alpha_s$ is a smooth Lie groupoid automorphism. Moreover, we can restrict this action to $\mathbb{T}_H U$, for any open subset $U \subset M$.

\begin{definition} 
Given a Lie groupoid $G$, we denote by $\Omega_r=\bigsqcup_{p \in G} |\Omega|^1(T_p G^{r(p)})$  the density bundle tangent with respect to the $r$-fibers, where $|\Omega|^1(T_p G^{r(p)})$ means the 1-density on the tangent space at $p$ of the fiber  $G^{r(p)}$. We denote by $C_p^\infty(G,\Omega_r)$ the space of properly supported smooth sections of this bundle. 
\end{definition}



We now recall the definition of an $r$-fibered distribution. For more informations on this topic the reader can read the article of Lescure-Manchon-Vassout \cite{lescure2017convolution}, see in particular theorem 2.1 (Schwartz kernel theorem for submersions)  and proposition 2.7 in \cite{lescure2017convolution}, and the articles of vEY \cite{Yuncken2019groupoidapproach},\cite{Yuncken2017tangentgroupoidfilteredmanifold}.
\begin{definition} \label{145} 
Given $G$ a Lie groupoid, we denote by $\mathcal{E}_r'(G)$ the set of properly supported $r$-fibered distributions on $G$, namely:
$$\mathcal{E}_r'(G):=\{u : C^\infty(G) \rightarrow C^\infty(\Gu),\text{ continuous and } C^\infty(\Gu)\text{-linear} \},$$
where the $C^\infty(\Gu)$-structure on $C^\infty(G)$ is induced by  the pull-back of the range map $r : G \rightarrow \Gu$. This means for all $ f \in C^\infty(G)$ and $ a \in C^\infty(\Gu)$ we have:
\begin{equation}
\langle u,r^*a.f \rangle= a \langle u,f \rangle.
\end{equation}
\end{definition}
Recall that a pseudo-differential operator $P$ on a manifold $M$ is said to be properly supported when its Schwartz kernel $k \in D'(M \times M)$ is a properly supported distribution in the sense of the pair groupoid $M \times M$ that is $r,s : \supp(k) \in M \times M \rightarrow M$ are proper maps. The Debord-Skandalis action is a key element in the following theorem concerning the classical pseudodifferential calculus, see \cite{Yuncken2019groupoidapproach}.
\begin{theoreme}{  vEY (2017).} \label{141}
Any properly supported classical pseudo-differential operator $P$ on a  trivially filtered  smooth manifold $M$ is the restriction to time $t=1$ of an properly supported $r$-fibered distribution $\mathbb{P}$, essentially homogeneous for the Debord-Skandalis action, see Definition \ref{121}, modulo properly supported smooth sections of the density bundle on the $r$-fibers of $\TM$. This means:
$$ \qquad \forall ~ s \in \R_+^*, ~~ \alpha_{s^*} \mathbb{P} - s^m\mathbb{P} \in C_p^\infty(\TM,\Omega_r).$$
\end{theoreme}
This theorem and its converse, which show that the operators of the calculus of vEY, see Definition \ref{144} below, are precisely the classical polyhomogeneous pseudodifferential operators, were the anchorage point from where vEY defined their pseudodifferential calculus on a filtered manifold. Here, we will show the analogous result for the BG calculus on a model Heisenberg manifold $M$, see Theorem \ref{157}. The following definition is extracted from \cite{Yuncken2017tangentgroupoidfilteredmanifold}.
\begin{definition}  \label{144} 
Given $M$ a filtered manifold, an element $k \in \mathcal{E}_r'(M \times M)$ is called an $H$-pseudodifferential kernel of order $\leq m$ if $k=\mathbb{P}|_{t=1}$ for some $\mathbb{P} \in \mathcal{E}_r'(\THM)$ essentially homogeneous for the Debord-Skandalis action on the filtered tangent groupoid, see Definition \ref{121}, modulo the properly supported smooth sections of the density bundle $\Omega_r$, meaning :
$$(\blacksquare)  \qquad \forall ~ s \in \R_+^*, ~~ \alpha_{s^*} \mathbb{P} - s^m\mathbb{P} \in C_p^\infty(\THM,\Omega_r).$$
We denote by $\bold{\Psi}^m(\THM)$ the set of essentially homogeneous properly supported $r$-fibered distribution on the filtered tangent groupoid $\mathbb{P} \in \mathcal{E}_r'(\THM)$ satisfying $(\blacksquare)$. 
The properly supported operators $P : C_c^\infty(M) \rightarrow D'(M)$ on a filtered manifold $M$ with such a Schwartz kernel $k$ are called  $H$-pseudodifferential operators. Such elements form  the set $\bold{\Psi}^m(\THM)|_{t=1}$.
\end{definition}
Morally, the $H$-tangent groupoid $\THM$ adds a supplementary dimension to the manifold, namely the $\R$-axis, to glue together the pair groupoid $M \times M$ (encoding the information of the kernel of a pseudodifferential operator) and the osculating groupoid (encoding the information of the co-symbol, that is the Fourier transform of the symbol of the pseudodifferential operator). This is analogous to the extra dimension which appeared in Theorem \ref{2}. 

In order to prove the equivalence of the calculi of BG and vEY, the remark after Theorem \ref{108} indicates that we need to consider the relation between the following two actions of $\R_+^*$ on $M \times \R^{d+1} \times \R$:
\begin{equation} \label{122}
\beta_s : M \times \R^{d+1} \times \R  \rightarrow M \times \R^{d+1} \times \R ,(x,v,t) \mapsto (x,\delta_s(v),st)=(x,\widetilde{\delta}_s(v,t)).
\end{equation}
\begin{equation} \label{142}
\widetilde{\alpha}_{s}:=(\Exp)^{-1} \circ \alpha_s \circ \Exp : M \times \R^{d+1} \times \R \rightarrow M \times \R^{d+1} \times \R,
\end{equation}
where $\alpha_s$ denotes the Debord-Skandalis action on $\mathbb{T}_H M$, see Definition \ref{121}.
The Fourier transform intertwines these two actions in the following way, see section 7.3 in \cite{Yuncken2019groupoidapproach}.
\begin{proposition} \label{123}
One can show that
$$\widetilde{\alpha}_{s}(x,v,t)=(x,\delta_s(v),s^{-1}t).$$ 
Moreover, the Fourier transform in the second variable $\mathcal{F}_2$ intertwines the actions $\beta_s$ and $\widetilde{\alpha}_{s}$,
$$\mathcal{F}_2 \circ \widetilde{\alpha}_{s*} \circ \mathcal{F}_2^{-1} = \beta_s^*.$$
\end{proposition}
We will need the following definition.

\begin{definition} \label{146} 
Given $U$ an open subset of $\R^n$. We set:
$$\mathcal{S}_G'(U \times \R^d):=\Hom_{C^\infty(U)}(S_G(U \times \R^d),C^\infty(U)) \cong C^\infty(U,S'(\R^{d})),$$
where $\Hom_{C^\infty(U)}$ denotes linearity with respect to $C^\infty(U)$.
\end{definition}

\begin{lemma} \label{124}
Given $U$ an open subset of $\R^n$ and $u \in \mathcal{H} \mathcal{S}^m( U \times \R^d)$, we have:
$$\mathcal{F}^{-1}(v \mapsto u(x,v)) \in C^\infty(U \times (\R^d \setminus \{ 0 \}) ) \cap \mathcal{S}_G'(U \times \R^d).$$
\end{lemma}

\begin{proof}
If $u $ is of the form:
$$u(x,v)=\chi(v)u'(x,v),$$
where $u' \in \mathcal{H}^m(U \times \R^d)$ and $\chi \in C_c^\infty(\R^d)$ is a smooth cut-off function which is 0 on a compact neighbourhood of 0 and 1 near infinity, the result is well-known.
More generally, thanks to Theorem \ref{1}, one can write the now familiar decomposition for any compact $K$ containing 0:
\begin{equation} \label{126}
u:=\chi_K u' + u'' ,
\end{equation}
where, $\chi_K \in C_c^\infty(\R^d)$, $u' \in \mathcal{H}^m(U \times \R^{d} )$ and $u'' \in \mathcal{S}_G(U \times \R^{d})$. The result follows.
 \end{proof}

\begin{theoreme} \label{157}
Let $M=\R^{d+1}$ with the model Heisenberg structure from \eqref{101}. Then

\begin{equation}
\bold{\Psi}^m_{\mathcal{V}}(M)=\bold{\Psi}^m(\THM)|_{t=1}.
\end{equation}

\end{theoreme}
\begin{proof}
We start with the inclusion $\bold{\Psi}^m_{\mathcal{V}}(M) \subset \bold{\Psi}^m(\THM)|_{t=1}$, where we remind the reader that the last set is defined in Definition \ref{144}.

Let $Op(q) \in \bold{\Psi}^m_{\mathcal{V}}(M)$ be a pseudodifferential operator in the BG calculus, where its symbol $q=\overline{\sigma}^*f$ and its Schwartz kernel $k$ are as in  Definition \ref{112}. Let $u \in \mathcal{H} \mathcal{S}^m(M \times \R^{d+1} \times \R)$ be the extension of $f$ as mentioned in the remark after Theorem \ref{108}. Put  $\kgt(y,\xi,t)=\mathcal{F}^{-1} \Big(v \mapsto u(y,v,t) \Big)(\xi)$ and $\kg=\Exp_*(\kgt)$. By Lemma \ref{124}, we have
\begin{equation}
\kgt \in C^\infty(M \times \R,  \mathcal{S}'( \R^{d+1})).
\end{equation}
Note also that  $\kgt|_{t=1}=\mathcal{F}_2^{-1}(u|_{t=1})=\mathcal{F}_2^{-1}(f)=\tilde{k}$, by Theorem \ref{108}. Using the exponential chart, this is  equivalent to $\kg|_{t=1}=k$.

Take a cut-off function $\chi \in C_p^\infty(M \times \R^{d+1} \times \R)$ which is equal to 1 on $M \times \{0 \} \times \R$. We define the following $r$-fibered distributions:
\begin{equation} \label{131}
\kgt_0:= \chi \kgt,~~ \kg_0=\Exp_*(\kgt_0).
\end{equation}
We are going to show that for all $s>0$:
\begin{equation} \label{133}
\tilde{\alpha}_{s^*} \kgt_0 - s^m \kgt_0 \in  C_p^\infty(M \times \R^{d+1} \times \R,\Omega_r).
\end{equation}
To do this, first compute the following for all $s>0$:
\begin{equation} \label{132}
\tilde{\alpha}_{s^*} \kgt_0 - s^m \kgt_0 =\underbrace{ \tilde{\alpha}_{s^*}(\kgt_0 - \kgt ) - s^m (\kgt_0 - \kgt)}_{(\spadesuit)}+  \underbrace{\tilde{\alpha}_{s^*}\kgt - s^m \kgt}_{ (\blacksquare)}.
\end{equation}
The term $(\spadesuit)$ is $C^\infty(M \times \R^{d+1} \times \R,\Omega_r)$ because of Lemma \ref{124} and of the calculation  $\kgt_0 - \kgt=(\chi -1  ) \kgt$. Now we deal with term $(\blacksquare)$. We have $\beta_s^*u - s^m u \in \mathcal{S}_G(M \times \R^{d+1} \times \R)$ for all $s>0$, by hypothesis on $u$.
Using Proposition \ref{123} this implies:
\begin{equation} \label{159}
\tilde{\alpha}_{s^*} \kgt - s^m \kgt \in \mathcal{S}_G(M \times \R^{d+1} \times \R) \subset C^\infty(M \times \R^{d+1} \times \R,\Omega_r).
\end{equation}
This prove that the term $(\blacksquare)$ on the right side of \eqref{132} is smooth and therefore so is the term on the left of \eqref{132}. Moreover $\kg_0$ is properly supported since $supp(\kgt_0) \subset supp(\chi)$. This proves \eqref{133}.

Finally, $( \kg - \kg_0 ) |_{t=1} \in  \Psi^{- \infty}(M)$ because this operator has a smooth kernel on $M \times M$. We therefore have
\begin{equation}
k=\kg_0|_{t=1}+ (\kg - \kg_0 ) |_{t=1},
\end{equation} 
which is what we needed to prove.
Now we prove the other inclusion $\bold{\Psi}^m(\THM)|_{t=1} \subset \bold{\Psi}^m_{\mathcal{V}}(M)$.

Let $P \in \bold{\Psi}^m(\THM)|_{t=1}$ with Schwartz kernel $k$ so that there exists $\kg \in \bold{\Psi}^m(\THM)$ that restricts at time $t=1$ to $P$.  Following Proposition 42 of \cite{Yuncken2019groupoidapproach} one may even choose $\kg$ to be homogeneous on the nose of order $m$ outside $[-1,1]$. 
 Define $\kgt=(\Exp)_*^{-1}(\kg)$ and $u=\mathcal{F}_2(\kgt)$.  We have that $\alpha_{s*} \kg - s^m \kg \in C_p^\infty(\THM,\Omega_r)$ and is $0$ outside $[-1,1]$. This implies that:
\begin{equation}
\tilde{\alpha_{s^*}} \kgt - s^m \kgt \in C_p^\infty(M \times \R^{d+1} \times \R,\Omega_r), ~~ and ~ is ~ 0 ~ outside ~ [-1,1].
\end{equation}
By Proposition \ref{123}, this leads to:
\begin{equation}
\beta_s^*u - s^m u \in C^\infty(M \times \R,\mathcal{S}(\R^{d+1})) , ~~ and ~ is ~ 0 ~ outside ~ [-1,1].
\end{equation}
Then the fact that this last expression is Schwartz in $\xi\in  \R^{d+1}$ and the condition on $t$ together imply that $u \in \mathcal{H} \mathcal{S}^m(M \times \R^{d+1} \times \R)$ and therefore, thanks to Theorem \ref{2}, $u|_{t=1} \in S_{phg,G}^m(M \times \R^{d+1})$. 
By the commuting diagramme following Theorem \ref{108}, we have:
\begin{align}
\mathcal{F}_2^{-1} \Big( \overline{\sigma}^{*} u|_{t=1}  \Big)(x,y) & =(\tau \circ m)_* k(x,y) \\
&= k(x,x-y),
\end{align}
so:
\begin{equation}
\mathcal{F}_2^{-1} \Big( \overline{\sigma}^{*} u|_{t=1}  \Big)(x,x-y)=k(x,y).
\end{equation}
By Definition \ref{112} we are done by putting $q(x,\xi):=\overline{\sigma}^{*} u|_{t=1}(x,\xi)$.

\end{proof}


\subsection{Application to contact manifolds and foliations}

Theorem \ref{157} can be easily extended to contact manifolds and foliations of codimension 1 which are both examples of Heisenberg manifolds. Note that CR-manifolds, appearing in the papers \cite{folland1974estimatesHeisenberg}, \cite{folland1973fundamentalsolutionsubOp} of Folland-Stein and in other papers from the 70's concerning the $\overline{\partial_b}$ Von-Neumann problem (\cite{folland1972tangentialCR}, \cite{folland2016NeumannPbCRcomplex}), are particular cases of contact manifolds. Such applications strongly motivate the desire of exploring and understanding the Heisenberg calculus and subelliptic operators on such manifolds. 

Recall that a continuous linear operator $P$ on $\mathcal{E}'(M)$ is called \emph{pseudolocal} if $\mathrm{singsupp}(Pf) \subseteq \mathrm{singsupp}(f)$ for every $f\in\mathcal{E}'(M)$, where $suppsing$ denotes the singular support of a distribution.  
The two definitions of the Heisenberg calculus in  \cite{Beals2016Heisenbergcalculus} and \cite{Yuncken2019groupoidapproach} are local in the following sense: a continuous linear operator $P$ on $\mathcal{E}'(M)$ is a Heisenberg pseudodifferential operator (in the sense of either \cite{Beals2016Heisenbergcalculus} or \cite{Yuncken2019groupoidapproach}) if and only if it is pseudolocal and for any open set $U$ and smooth function $\chi$ supported in $U$, the operator $M_\chi \circ P \circ M_\chi$ is a Heisenberg pseudodifferential operator on $U \subset M$ (in the sense of \cite{Beals2016Heisenbergcalculus} or \cite{Yuncken2019groupoidapproach}, respectively).  Here, $M_\chi$ denotes the operator of multiplication by $\chi$.  

Let $M$ be a contact manifold of dimension $2d+1$, that is a smooth manifold with a 1-form $\omega \in \Omega^1(M)$ such that $w \wedge \underbrace{d \omega \wedge ... \wedge d \omega}_{2d ~ times}$ is nowhere zero. We write $\mathcal{V}=\mathrm{ker}(\omega)$ for the contact hyperplane bundle. By Darboux's Theorem \cite[Appendix 4]{Arnold1978mechanics}, around any point $p\in M$ we can find an open neighbourhood $U$ and a contactomorphism $\varphi: U \to \mathbb{H}_d$, to the Heisenberg group with its standard contact structure. This means that around any point $p \in M$ and for any vector field $X \in \Gamma^\infty(\mathcal{V})$ :
\begin{equation}
\varphi_*(X)=\sum_{i=1}^{2d} f_i X_i,
\end{equation}
where $X_0,...,X_{2d}$ are the model fields for $\mathbb{H}_d$ and $f_i$ are smooth functions on $\R^{2d+1}$.

Let $P$ be a pseudolocal operator on $\mathcal{E}'(M)$.   Let $\chi$ be a bump function at $p$ with support in $U$.  Using Theorem \ref{157}, the push-forward of $M_\chi \circ P \circ M_\chi$ to   $\mathbb{H}_d$ via $\varphi$ is in the Beals-Greiner calculus if and only if it is in the groupoid calculus.  Since $p$ was arbitrary, we obtain Theorem \ref{contact_manifold}.

The same general argument applies to foliations.  If $M$ is a $d+1$-dimensional manifold equipped with a codimension $1$ foliation, then at every point $p$ we can find a foliation chart $\varphi : U\to \R^{d+1}$ from a neighbourhood $U$ of $p$ to the abelian group $\R^{d+1}$ with its standard folation.  Applying the same reasoning as above, we obtain the following theorem.

\begin{theoreme}
\label{foliations}
Let $M$ be a manifold equipped with a codimension $1$ foliation.  Beals and Greiner's algebra of Heisenberg pseudodifferential operators on $M$ \cite{Beals2016Heisenbergcalculus} coincides with the algebra of pseudodifferential operators defined via the Heisenberg tangent groupoid of $M$ in \cite{Yuncken2019groupoidapproach}. That is, we again have:
\begin{equation}
\bold{\Psi}^m_{\mathcal{V}}(M)=\bold{\Psi}^m(\THM)|_{t=1}.
\end{equation}

\end{theoreme}

\bibliographystyle{plain}


\bibliography{bibarticle}

\end{document}